\documentclass[12pt]{amsart}
\usepackage{geometry}   
\usepackage[colorlinks,citecolor = red, linkcolor=blue,hyperindex]{hyperref}
\usepackage{euscript,eufrak,verbatim, mathrsfs}
\usepackage[psamsfonts]{amssymb}
\usepackage{bbm}
\usepackage{graphicx}
 \usepackage{float}
\usepackage{float, tikz}

 \usepackage{extarrows}
 \usepackage[all, cmtip]{xy}

\usepackage{upref, xcolor, dsfont}
\usepackage{amsfonts,amsmath,amstext,amsbsy, amsopn,amsthm}
\usepackage{enumerate}

\usepackage{url}

\usepackage{mathtools}
\usepackage{bookmark}

 \usepackage{euscript}
\usepackage{helvet}         
\usepackage{courier}        
\usepackage{type1cm}        
\usepackage{multicol}        
\usepackage[bottom]{footmisc}

\newtheorem{theorem}{Theorem}[section]
\newtheorem*{theorem*}{Theorem A}
\newtheorem{lemma}[theorem]{Lemma}

\newtheorem*{definition*}{Definition}
\newtheorem*{remark*}{\text{Re}mark}

\newtheorem*{observation*}{Observation}

\newtheorem*{assumption*}{Assumption}

\theoremstyle{definition}

\theoremstyle{remark}
\newtheorem{remark}{\text{Re}mark}[section]

\newcommand{\D}{\mathbb{D}}
\newcommand{\C}{\mathbb{C}}

\newcommand{\T}{\mathbb{T}}

\begin{document}

\title{Hypercontractive inequalities and Nikol'ski\u{\i}-type inequalities on weighted Bergman spaces}

\author
{Zipeng Wang}
\address{Zipeng Wang\\
	College of Mathematics and Statistics, Chongqing University\\
	Chongqing, 401331, China}
\email{zipengwang2012@gmail.com}

\author
{Kenan Zhang}
\address
{Kenan Zhang\\
	School of Mathematical Sciences, Fudan University\\
	Shanghai, 200433, China}
\email{knzhang21@m.fudan.edu.cn}

\thanks{}

\begin{abstract}
In this note, we obtianed hypercontractive inequalities between different weighted Bergman spaces. In addition, we establish Nikol'ski\u{\i}-type inequalities for weighted Bergman spaces with optimal constants.
\end{abstract}


\maketitle

\setcounter{equation}{0}

\section{Introduction}
Let $\T$ be the unit circle and we  identify the circle $\T$ with the interval $[0,2\pi)$. 
For $1\leqslant p<\infty$, $L^p(\T)$ is the Banach space of $L^p$-integrable function on the unit circle. Recall the Fourier coefficients $\hat{f}(n),n\in\mathbb{Z}$ of $f\in L^p(\T)$ is defined by
\[
\hat{f}(n)=\int_0^{2\pi}f(e^{i\theta})e^{-in\theta}d\theta,
\]
where $d\theta$ is the normalized Lebesgue measure on $[0,2\pi)$.
Then the Hardy space $H^p(\T)$ is the closed subspece of $L^p(\T)$: 
\[
	H^p(\T)=\left\{f\in L^p(\T)\left|\hat{f}(n)=0, \text{for each neagtive integer } n\right.\right\}.
\]
Let $f\in L^2(\T)$, then $f$ has a Fourier expansion
\[
	f(e^{i\theta})=\sum\limits_{n=-\infty}^{+\infty} \hat{f}(n) e^{in\theta}.
\]
For $0\leqslant r<1$, define the dialation operator
\[
	(T_r f)(e^{i\theta}):=\sum\limits_{n=-\infty}^{+\infty} \hat{f}(n) r^{|n|} e^{in\theta}.
\]
Observe that $\{T_r\}_{0\leqslant r<1}$ is a contractive operator semigroup on $L^p(\T)$.
The hypercontractive inequalities for $T_r$ is to determine the optimal range of $r$ such that $$T_r:L^p(\T)\to L^q(\T)$$ is a contraction. 
In his celebrated work, Weissler \cite{Weissler1980} established the following sharp criterion:
\[
	\|T_rf\|_{L^q(\T)}\leqslant \|f\|_{L^p(\T)}
\]
if and only if $r^2\leqslant (p-1)/(q-1)$.
Moreover, 
\[
	\|T_rf\|_{H^q(\T)}\leqslant \|f\|_{H^p(\T)}
\]
if and only if $r^2\leqslant p/q$. 

Let $\D$ be the unit disk and $dA=dxdy/\pi$ be the normalized area measure on $\mathbb{D}$.
Suppose that $0<p<\infty$ and $ 1<\alpha<\infty,$ the standard weighted Bergman space $A_\alpha^p(\mathbb{D})$ consists of analytic functions on the unit disk such that
\[
	||f||_{A_\alpha^p(\mathbb{D})}:=\left( \int_\mathbb{D} |f(z)|^p dA_\alpha(z) \right)^\frac{1}{p}<\infty,
\]
where 
$dA_\alpha(z)=(\alpha-1)(1-|z|^2)^{\alpha-2}dA(z).$
Recall each analytic function $f$ on $\D$ has the following Taylor expansion at $0$:
\[
	f(z)=\sum\limits_{n=0}^{+\infty} a_n z^n,\quad \forall z\in\D.
\]
Then the dilation operator can be formally written as
\[
	T_r f(z):=\sum\limits_{n=0}^{+\infty} a_n r^n z^n=f(rz),\quad\forall z\in\mathbb{D}.
\]
In a recent work Melentijević \cite{Petar2023}, he obtained the optimal hypercontractive constant $r_0=\sqrt{p/q}$ such that   
\[
||T_rf ||_{A_\alpha^q(\mathbb{D})} \leqslant ||f||_{A_\alpha^p(\mathbb{D})},\quad	\forall 0<r\leqslant r_0
\]
for $0<p\leqslant q<\infty$ and $q\geqslant 2$.
Melentijević's work 
is motivated by Janson’s hypercontractive problem of dialtion operators on spaces of analytic functions. For background and recent developments on Janson's problem, one can consult \cite{Kulikov2022} and \cite{Petar2023}.

The first goal of current paper is to extend  Melentijević's work to the case of two distinct weighted Bergman spaces $A_\alpha^p(\mathbb{D})$ and $A_\beta^q(\mathbb{D}).$

\begin{theorem}\label{main01}
	For $1<\alpha,\beta<\infty$ and $0<p\leqslant q<\infty$, let  $q\geqslant 2$ and $\beta p\leqslant\alpha q$. Then  we have 
	\begin{equation}\label{hypercontractive}
		||T_rf ||_{A_\beta^q(\mathbb{D})} \leqslant ||f||_{A_\alpha^p(\mathbb{D})}
	\end{equation}
	for any $ f\in A_\alpha^p(\mathbb{D})$ if and only if $ r^2\leqslant\beta p/\alpha q$.
\end{theorem}
\begin{remark}
It is known that $\lim\limits_{\alpha\to 1}\|f\|_{A_\alpha^p}=\|f\|_{H^p}$ for $f\in H^p$. 
Hence, Theorem \ref{main01} is a generalization of Janson's strong hypecontractive inequalities between Hardy spaces in \cite{Janson1983} for the case $0<p\leqslant q<\infty$ and $q\geqslant 2$. The proofs of Theorem \ref{main01} are inspried by some ideas and methods from \cite{Kulikov2022} and \cite{Petar2023}.
After our paper was uploaded to arXiv, Huang Yi and Zhang Jianyang informed us that they had also obtained similar results.
\end{remark}
Let $\mathbb{C}[z_1,\dots,z_n]$ be the space of complex polynomials on $\mathbb{C}^n$ and $\textbf{z}=(z_1,\dots,z_n)\in\C^n$
For a $m$-homogeneous polynomial $P(\textbf{z})\in\mathbb{C}[z_1,\dots,z_n]$, as a corollary of Weissler's hypecontractive work \cite{Weissler1980}, we have
$$||P||_{H^q(\mathbb{T}^n)}\leqslant\left(\sqrt{\frac{q}{p}}\right)^m ||P||_{H^p(\mathbb{T}^n)},$$
where
\[
	||P||_{H^p(\mathbb{T}^n)}=\left(\int_{\mathbb{T}^n}|P(\textbf{z})|^p dm(\textbf{z})\right)^\frac{1}{p}
\]
and $dm(\textbf{z})$ is the Haar measure on $\mathbb{T}^n$. 
For a complex polynomial $P$ in $\mathbb{C}[z_1,\dots,z_n]$,
we write 
\[
	P(\textbf{z})=\sum\limits_{\gamma\in\mathbb{N}^n}c_\gamma \textbf{z}^\gamma,
\]
where $\gamma=(\gamma_1,\dots,\gamma_n)\in\mathbb{N}^n$ is a multi-index and $\textbf{z}^\gamma=z_1^{\gamma_1}\cdots z_n^{\gamma_n}$. 
Let $|\gamma|=\sum\limits_{j=1}^n \gamma_j$. 
Then the degree of the polynomial $P$ is defined by
\[
	\text{deg}(P):=\max\{|\gamma|:c_\gamma\not=0\}.
\]
Let $m$ be the degree of the polynomial $P$. The famous Nikol'ski\u{\i}-type inequality \cite{Nikol1951,Nikol1975} states that for any complex polynomial $P$ in $\mathbb{C}[z_1,\dots,z_n]$ and $0<q<p<\infty$, 
\[
\|P\|_{H^q(T^n)}\leqslant C\|P\|_{H^p(\mathbb{T}^n)},
\]
where the constant $C=C(n,m,p,q)$.

Using hypercontractive inequalities of the Poisson transform between Hardy spaces \cite{Weissler1980}, Defant and Masty{\l}o \cite{Defant2016} 
obtained a dimension-free estimate of Nikol'ski\u{\i}-type inequality for Hardy spaces. In particular,  for $0<p<q<\infty$, the optimal constant in  Nikol'ski\u{\i}-type inequality for Hardy spaces is 
\[
	C(m,p,q)=\left(\sqrt{\frac{q}{p}}\right)^m.
\]

For $n\geqslant 1,0<p<\infty$ and $1<\alpha<\infty$, recall the weighted Bergman space $A_\alpha^p(\mathbb{D}^n)$ on the polydisc $\mathbb{D}^n$ is
\[
A_\alpha^p(\mathbb{D}^n)=\left\{f\in H(\mathbb{D}^n)\left| ||f||_{A_\alpha^p(\mathbb{D}^n)}^p:=\int_{\mathbb{D}^n}|f(z_1,\dots,z_n)|^pdA_\alpha(z_1)\cdots dA_\alpha(z_n)<\infty\right.\right\},
\]
where $H(\mathbb{D}^n)$ is the set of analytic functions on $\mathbb{D}^n$.
Inspried by ideas from \cite{Defant2016}, we have the following dimensional free Nikol'ski\u{\i}-type inequality for the weighted Bergman spaces with optimal constants.
\begin{theorem}\label{main02}
	For $1<\alpha,\beta<\infty$ and $0<p\leqslant q<\infty$, let $q\geqslant 2$ and $\beta p\leqslant\alpha q.$ Then for any positive integer $m,n$ and a complex polynomial $P\in\mathbb{C}[z_1,\dots,z_n]$ with degree $m$ 
	\begin{equation}\label{Nikol}
		||P||_{A_\beta^q(\mathbb{D}^n)}\leqslant\left(\sqrt{\frac{\alpha q}{\beta p}}\right)^m ||P||_{A_\alpha^p(\mathbb{D}^n)}.
	\end{equation}
	The constant $$C(\alpha,\beta,p,q)=\sqrt{\frac{\alpha q}{\beta p}}$$ is the best possible such that for each positive integers $n,m$ and $ P\in\mathbb{C}[z_1,\dots,z_n]$ with degree $m,$ the following inequality holds
	$$||P||_{A_\beta^q(\mathbb{D}^n)}\leqslant C(\alpha,\beta,p,q)^m ||P||_{A_\alpha^p(\mathbb{D}^n)}.$$
\end{theorem}

\vspace{0.2in}

\noindent \textbf{Acknowledgement.} 
We would like to express our gratitude to Professor Kunyu Guo for his insightful suggestions on the current manuscript.
The work was supported by NSFC of China (No. 12471116).

\section{The proof of Theorem \ref{main01}}
\subsection{The Necessity}
We start with a simple but useful lemma in \cite{Janson1983} to prove the necessity for Theorem \ref{main01}.
\begin{lemma}\label{nec01}
	For $1<\alpha<\infty, 0<p<\infty$ and a real number $\varepsilon,$ let $f(z)=1+\varepsilon z$. As the real number $\varepsilon$ approaches $0$, we have
	$$||f||_{A_\alpha^p(\mathbb{D})}=||1+\varepsilon z||_{A_\alpha^p(\mathbb{D})}=1+\frac{p}{4\alpha}\varepsilon^2+O(\varepsilon^3).$$
\end{lemma}
\begin{proof}
	Fix $z\in\mathbb{C}$, then
	\begin{align*}
		|1+\varepsilon z|^p
		&=\left(1+2\varepsilon \text{Re}{z}+\varepsilon^2|z|^2\right)^{\frac{p}{2}}\\
		&=1+\frac{p}{2}\left(2\varepsilon \text{Re}{z}+\varepsilon^2|z|^2\right)+\frac{1}{2}\cdot\frac{p}{2}\left(\frac{p}{2}-1\right)\left(2\varepsilon \text{Re}{z}+\varepsilon^2|z|^2\right)^2+O\left(\varepsilon^3\right).
	\end{align*}
	Integrating on the unit disk with respect to the measure $dA_\alpha$, one gets 
	\begin{align*}
		&\left(\int_\mathbb{D}|1+\varepsilon z|^p dA_\alpha(z)\right)^\frac{1}{p}\\
		=&1+\frac{1}{2}\varepsilon^2\int_\mathbb{D}|z|^2 dA_\alpha(z)+\frac{p-2}{8}\int_\mathbb{D}\left(2\varepsilon \text{Re}{z}+\varepsilon^2|z|^2\right)^2dA_\alpha(z)+O(\varepsilon^3)\\
		=&1+\frac{1}{2}\varepsilon^2\left(||z||_{A_\alpha^2(\mathbb{D})}^2+(p-2)||\text{Re}z||_{A_\alpha^2(\mathbb{D})}^2\right)+O(\varepsilon^3)\\
		=&1+\frac{p}{4\alpha}\varepsilon^2+O(\varepsilon^3).
	\end{align*}
This completes the proof of Lemma \ref{nec01}.
\end{proof}

For any $f\in A_\alpha^p(\mathbb{D})$ and $r\in[0,1)$, recall that 
$$
T_rf(z)=f(rz),\quad z\in\mathbb{D}.
$$
Then, by Lemma \ref{nec01},
\[
	||1+\varepsilon rz||_{A_\beta^q(\mathbb{D})}=1+\frac{q}{4\beta}\varepsilon^2r^2+O(\varepsilon^3).
\]
Assume the hypercontractive inequality \eqref{hypercontractive} holds, let $f(z)=1+\varepsilon z$, we have
\[
	\|1+\varepsilon rz\|_{A_\beta^q(\mathbb{D})}\leqslant \|1+\varepsilon z\|_{A_\alpha^p(\mathbb{D})},
\]
Hence,
\[
	1+\frac{q}{4\beta}\varepsilon^2r^2\leqslant 1+\frac{p}{4\alpha}\varepsilon^2+O(\varepsilon^3).
\]
As $\varepsilon$ approaches $0$, it follows that 
\[
	r^2\leqslant\frac{\beta p}{\alpha q}.
\]
This achieves the necessary part of Theorem \ref{main01}. 
 
\subsection{The Sufficienty}
	For $0<p\leqslant q<\infty,$ let 
	\[
		\beta'=\frac{q}{p}\cdot\alpha,
	\] then 
	\[
		\frac{\beta'}{\alpha}=\frac{q}{p}\geqslant 1.
	\]
	 By Kulikov's inequality \cite[Corollary 1.3]{Kulikov2022}, it is easy to see
	$$||f||_{A_{\beta'}^q(\mathbb{D})} \leqslant ||f||_{A_\alpha^p(\mathbb{D})}.$$
	Since the $A_\beta^{q}(\mathbb{D})$ norm of $T_rf$ is decreasing as $r$ is decreasing, it suffices to prove that there holds
	$$||T_rf||_{A_\beta^{q}(\mathbb{D})} \leqslant ||f||_{A_{\beta'}^q(\mathbb{D})}$$
	for $r=\sqrt{\beta p/\alpha q}.$
	It is equivalent to
	$$(\beta-1)\int_\mathbb{D}|f(rz)|^q (1-|z|^2)^{\beta-2}dA(z)\leqslant (\beta'-1)\int_\mathbb{D}|f(z)|^q (1-|z|^2)^{\beta'-2}dA(z).$$
	By the polar coordinates, we need to prove
	\begin{align*}
		&(\beta-1)\int_0^1\rho(1-\rho^2)^{\beta-2}\left(\int_0^{2\pi}|f(r\rho e^{i \theta})|^qd\theta\right)d\rho\\
		\leqslant&(\beta'-1)\int_0^1\rho(1-\rho^2)^{\beta'-2}\left(\int_0^{2\pi}|f(\rho e^{i \theta})|^qd\theta\right)d\rho.
	\end{align*}
	Denote $\rho=\sqrt{y}$ and $$\Phi(y)=\int_0^{2\pi}|f(\sqrt{y} e^{i \theta})|^qd\theta,$$ it is sufficient to verify
	\[
		(\beta-1)\int_0^1(1-y)^{\beta-2}\Phi(r^2y)dy\leqslant(\beta'-1)\int_0^1(1-y)^{\beta'-2}\Phi(y)dy
	\]
	when $r=\sqrt{\beta p/\alpha q}=\sqrt{\beta/\beta'}.$
Since $q\geqslant 2$, by \cite[Lemma 1]{Petar2023}, $\Phi$ is $C^2$ smooth and $\Phi''(y)\geqslant 0.$ 
	Using integration by parts, we have
	\begin{align*}
		(\beta-1)\int_0^1(1-y)^{\beta-2}\Phi(r^2y)dy
		&=r^2\int_0^1(1-y)^{\beta-1}\Phi'(r^2y)dy\\
		&=\frac{r^2}{\beta}\Phi'(0)+\frac{r^4}{\beta}\int_0^1(1-y)^\beta\Phi''(r^2y)dy\\
		&=\frac{1}{\beta'}\Phi'(0)+\frac{r^2}{\beta'}\int_0^1(1-y )^\beta\Phi''(r^2y)dy\\
		&=\frac{1}{\beta'}\Phi'(0)+\frac{1}{\beta'}\int_0^{r^2}\left(1-\frac{y}{r^2}\right)^\beta\Phi''(y)dy,
	\end{align*}
	and 
	\begin{align*}
		(\beta'-1)\int_0^1(1-y)^{\beta'-2}\Phi(y)dy
		&=\int_0^1(1-y)^{\beta'-1}\Phi'(y)dy\\
		&=\frac{1}{\beta'}\Phi'(0)+\frac{1}{\beta'}\int_0^1(1-y )^{\beta'}\Phi''(y)dy.
	\end{align*}
	Then it suffices to prove
	\[
		\int_0^{r^2}\left(1-\frac{y}{r^2}\right)^\beta\Phi''(y)dy\leqslant\int_0^1(1-y )^{\beta'}\Phi''(y)dy.
	\]
Note that $\beta'/\beta=\alpha q/\beta p\geqslant 1$, the function $g(y)=(1-y)^{\beta'/\beta}$ is convex.
Hence, for all $0\leqslant y\leqslant\beta/\beta'=r^2$, we have
\[
	g(y)\geqslant g(0)+g'(0)y=1-\frac{y}{r^2},
\]
Note that $\Phi''(y)\geqslant 0$, we have
\begin{align*}
	\int_0^{r^2}\left(1-\frac{y}{r^2}\right)^\beta\Phi''(y)dy
	&\leqslant\int_0^{r^2}(1-y )^{\beta'}\Phi''(y)dy\\
	&\leqslant\int_0^1(1-y )^{\beta'}\Phi''(y)dy.
\end{align*}
This completes the whole proof.

\section{Proof of Theorem \ref{main02}}
Before proving Nikol'ski\u{\i}-type inequality for the weighted Bergman spaces, we need a hypercontractive inequality for high dimensions.
Let $n$ be a positive integer and $f\in A_\alpha^p(\mathbb{D}^n)$. Assume that $\textbf{z}=(z_1,\dots,z_n)\in\mathbb{D}^n$ and $\textbf{r}=(r_1,\dots,r_n)$ with $0\leqslant r_i<1$ for $i=1,\dots,n.$ 
Define an operator
\[
(\textbf{T}_\textbf{r} f)(\textbf{z})=f(r_1z_1,\dots,r_nz_n), \quad \textbf{z}\in\mathbb{D}^n.
\]
The following lemma is a corollary of Theorem \ref{main01}.

\begin{lemma}\label{Bayart}
	For $1<\alpha,\beta<\infty$ and $0<p\leqslant q<\infty$ , let $q\geqslant 2$, $\beta p\leqslant\alpha q$, and $\textbf{r}=(r_1,\dots,r_n)$ with $0\leqslant r_i<1$ for $i=1,\dots,n$. Then we have 
	\begin{equation}\label{high-dimensions}
		||\textbf{T}_\textbf{r} f||_{A_\beta^q(\mathbb{D}^n)} \leqslant ||f||_{A_\alpha^p(\mathbb{D}^n)}
	\end{equation}
	for any $f\in A_\alpha^p(\mathbb{D}^n)$ if and only if $ r_i^2\leqslant\beta p/\alpha q,$ for all $i=1,2,\dots,n.$
\end{lemma}
\begin{proof}
We prove the necessity first.
Assume that 
\[
	||\textbf{T}_\textbf{r} f||_{A_\beta^q(\mathbb{D}^n)} \leqslant ||f||_{A_\alpha^p(\mathbb{D}^n)}
\]
for any $f\in A_\alpha^p(\mathbb{D}^n)$.
For each $1\leqslant i\leqslant n$, let $g\in A_\alpha^p(\mathbb{D})$ and
\[
	f(\textbf{z})=f(z_1,\dots, z_i,\dots, z_n):=g(z_i),
\]
then $f\in A_\alpha^p(\mathbb{D}^n)$ and 
\[
	\|f\|_{A_\alpha^p(\mathbb{D}^n)}=\|g\|_{A_\alpha^p(\mathbb{D})}.
\] 
Moreover,
\[
	(\textbf{T}_\textbf{r} f)(\textbf{z})=(T_{r_i} g)(z_i).
\]
Then we have
\[
	||T_{r_i} g||_{A_\beta^q(\mathbb{D})} \leqslant ||g||_{A_\alpha^p(\mathbb{D})}.
\]
By Theorem \ref{main01}, we get $ r_i^2\leqslant\beta p/\alpha q$ for each $ i=1,2,\dots,n.$

We proceed by induction on $n$ to prove the sufficienty.
Suppose that $r_i^2\leqslant\beta p/\alpha q$ for all $ i=1,2,\dots,n$.
First, the inequality \eqref{high-dimensions} is true for $n=1$ by Theorem \ref{main01}. 
Assume that lemma is valid for $n=k-1$.
We will prove the inequality \eqref{high-dimensions} still holds when $n=k$.
Recall that 
\[
	||\textbf{T}_\textbf{r}f||_{A_\beta^q(\mathbb{D}^k)}
	=\left(\int_{\mathbb{D}^{k}}|f(r_1z_1,\dots,r_{k-1}z_{k-1},r_kz_k)|^q dA_\beta(z_1)\cdots dA_\beta(z_{k-1})dA_\beta(z_k)\right)^{\frac{1}{q}}.
\]
By Fubini Theorem and Theorem \ref{main01}, we have
\begin{align*}
	||\textbf{T}_\textbf{r}f||_{A_\beta^q(\mathbb{D}^k)}
	=&\left(\int_{\mathbb{D}^{k-1}}\left(\int_\mathbb{D}|f(r_1z_1,\dots,r_{k-1}z_{k-1},r_kz_k)|^q dA_\beta(z_k)\right)dA_\beta(z_1)\cdots dA_\beta(z_{k-1})\right)^{\frac{1}{q}}\\
	\leqslant&\left(\int_{\mathbb{D}^{k-1}}\left(\int_\mathbb{D}|f(r_1z_1,\dots,r_{k-1}z_{k-1},z_k)|^p dA_\beta(z_k)\right)^\frac{q}{p} dA_\beta(z_1)\cdots dA_\beta(z_{k-1})\right)^{\frac{1}{q}}.
\end{align*}
Since $0<p\leqslant q<\infty$, by Minkowski's inequality and the induction hypothesis, we have
\begin{align*}
	||\textbf{T}_\textbf{r}f||_{A_\beta^q(\mathbb{D}^k)}\leqslant&\left(\int_\mathbb{D}\left(\int_{\mathbb{D}^{k-1}}|f(r_1z_1,\dots,r_{k-1}z_{k-1},z_k)|^q dA_\beta(z_1)\cdots dA_\beta(z_{k-1})\right)^\frac{p}{q} dA_\beta(z_k)\right)^{\frac{1}{p}}\\
	\leqslant&\left(\int_\mathbb{D}\left(\int_{\mathbb{D}^{k-1}}|f(z_1,\dots,z_{k-1},z_k)|^p dA_\beta(z_1)\cdots dA_\beta(z_{k-1})\right) dA_\beta(z_k)\right)^{\frac{1}{p}}\\
	=&||f||_{A_\alpha^p(\mathbb{D}^k)}.
\end{align*}
This completes the proof of Lemma \ref{Bayart}.
\end{proof}

As a direct corollary of Lemma \ref{Bayart}, the inequality \eqref{Nikol} holds for $m$-homogeneous polynomials. For given $n\in\mathbb{N}$, we define a new norm for $Q(\textbf{z},w)\in\mathbb{C}[z_1,\dots,z_n,z_{n+1}].$ Denote
$$||Q||_{A^p_\alpha(\mathbb{D}^n,\textbf{z})\times H^p(\mathbb{T},w)}^p=\int_\mathbb{T}\int_{\mathbb{D}^n}|Q(\textbf{z},w)|^p dA_\alpha(\textbf{z})dw,$$ 
where $dA_\alpha(\textbf{z})=dA_\alpha(z_1)\cdots dA_\alpha(z_n)$
and $dw$ is the normalized arc length on the unit circle $\mathbb{T}.$ 
The following homogeneous technique of general complex polynomials is known. However, we cannot locate a suitable reference. For the sake of completeness, we include a short proof.

\begin{lemma}\label{Defant}
	For a given $n\in\mathbb{N}$, let $\textbf{z}=(z_1,\dots,z_n)\in\mathbb{D}^n,\ P(\textbf{z})\in\mathbb{C}[z_1,\dots,z_n]$ be a complex polynomial with degree $m$, then there exists an $m$-homogeneous polynomial $Q(\textbf{z},w)\in\mathbb{C}[z_1,\dots,z_n,z_{n+1}]$ such that
$$||P||_{A^p_\alpha(\mathbb{D}^n)}=||Q||_{A^p_\alpha(\mathbb{D}^n,\textbf{z})\times H^p(\mathbb{T},w)}.$$

\end{lemma}
\begin{proof}
For a polynomial $P$ in $\mathbb{C}[z_1,\dots,z_n]$,
write
\[
	P(\textbf{z})=\sum\limits_{\gamma\in\mathbb{N}^n}c_\gamma \textbf{z}^\gamma,
\]
where $\gamma=(\gamma_1,\dots,\gamma_n)\in\mathbb{N}^n$ is a multi-index. 
Let $\text{deg}\ P=m$ and
\[
	Q(\textbf{z},w)=\sum\limits_{\gamma\in\mathbb{N}^n}c_\gamma \textbf{z}^\gamma w^{m-|\gamma|}\in\mathbb{C}[z_1,\dots,z_n,z_{n+1}],
\]
where $|\gamma|=\gamma_1+\dots+\gamma_n.$ Then $Q(\textbf{z},w)$ is $m$-homogeneous.
For $0<p<\infty,$ recall that
\begin{align*}
	||Q||_{A^p_\alpha(\mathbb{D}^n,\textbf{z})\times H^p(\mathbb{T},w)}^p&=\int_\mathbb{T}\int_{\mathbb{D}^n}|Q(\textbf{z},w)|^p dA_\alpha(\textbf{z})dw\\
	&=\int_\mathbb{T}\int_{\mathbb{D}^n}\left|\sum\limits_{\gamma\in\mathbb{N}^n}c_\gamma \textbf{z}^\gamma w^{m-|\gamma|}\right|^p dA_\alpha(\textbf{z})dw.
\end{align*}
Since $w\in\T$, let $\xi=(z_1 w^{-1},\dots,z_n w^{-1}),$ we have
\begin{align*}
	||Q||_{A^p_\alpha(\mathbb{D}^n,\textbf{z})\times H^p(\mathbb{T},w)}^p
	&=\int_\mathbb{T}\int_{\mathbb{D}^n}\left|\sum\limits_{\gamma\in\mathbb{N}^n}c_\gamma (z_1w^{-1})^{\gamma_1}\dots(z_nw^{-1})^{\gamma_n}\right|^p dA_\alpha(\textbf{z})dw\\
	&=\int_\mathbb{T}\int_{\mathbb{D}^n}\left|\sum\limits_{\gamma\in\mathbb{N}^n}c_\gamma \xi_1^{\gamma_1}\dots \xi_n^{\gamma_n}\right|^p dA_\alpha(\xi)dw\\
	&=||P||_{A_\alpha^p(\mathbb{D}^n)}^p.
\end{align*}
This completes the proof.
\end{proof}

Now we are ready to prove Theorem \ref{main02}.
Assume that
\[
	P(\textbf{z})=\sum\limits_{\gamma\in\mathbb{N}^n}c_\gamma \textbf{z}^\gamma\in\mathbb{C}[z_1,\dots,z_n]
\]
is a polynomial.
We divide the problem into two cases.
For the case $1<\alpha<\beta<\infty.$ 
Let
\[
	q'=\frac{\alpha}{\beta} q,
\]
then $0<p\leqslant q'<q<\infty$, $q'/\alpha=q/\beta$.
By Kulikov's inequality \cite[Corollary 1.3]{Kulikov2022}, one gets
\[
	||P||_{A_\beta^q(\mathbb{D}^n)}\leqslant||P||_{A_\alpha^{q'}(\mathbb{D}^n)}.
\]
Then it suffices to prove 
\[
	||P||_{A_\alpha^{q'}(\mathbb{D}^n)}\leqslant \left(\sqrt{\frac{\alpha q}{\beta p}}\right)^m ||P||_{A_\alpha^p(\mathbb{D}^n)}=\left(\sqrt{\frac{q'}{p}}\right)^m ||P||_{A_\alpha^p(\mathbb{D}^n)}
\]
Let $r_0=\sqrt{p/q'}<1$, $\textbf{r}_0=(r_0,\dots,r_0)$ and 
\[
	Q(\textbf{z},w)=\sum\limits_{\gamma\in\mathbb{N}^n}c_\gamma \textbf{z}^\gamma w^{m-|\gamma|}\in\mathbb{C}[z_1,\dots,z_{n+1}]
\]
be a $m$-homogeneous polynomial. 
Since for each $w\in\T$, $Q(\textbf{z},w)\in A_\alpha^p(\mathbb{D}^n)$. 
By Lemma \ref{Bayart} and Minkowski's inequality, we have
\begin{align*}
	||\textbf{T}_{\textbf{r}_0}Q||_{A^{q'}_\alpha(\mathbb{D}^n,\textbf{z})\times H^{q'}(\mathbb{T},w)}
	&=\left(\int_\mathbb{T}\int_{\mathbb{D}^n}|Q(\textbf{r}_0\textbf{z},r_0w)|^{q'} dA_\beta(\textbf{z})dw\right)^\frac{1}{q'}\\
	&\leqslant\left(\int_\mathbb{T}\left(\int_{\mathbb{D}^n}|Q(\textbf{z},r_0w)|^{p} dA_\beta(\textbf{z})\right)^{\frac{q'}{p}}dw\right)^\frac{1}{q'}\\
	&\leqslant\left(\int_{\mathbb{D}^n}\left(\int_\mathbb{T}|Q(\textbf{z},r_0w)|^{q'} dw\right)^{\frac{p}{q'}}dA_\beta(\textbf{z})\right)^\frac{1}{p}
\end{align*}
Since for each $\textbf{z}\in\D$, $Q(\textbf{z},w)\in H^p(\T)$.
By Weissler's work \cite{Weissler1980}, one gets
\begin{align*}
	||\textbf{T}_{\textbf{r}_0}Q||_{A^{q'}_\alpha(\mathbb{D}^n,\textbf{z})\times H^{q'}(\mathbb{T},w)}
	&\leqslant\left(\int_\mathbb{T}\int_{\mathbb{D}^n}|Q(\textbf{z},w)|^{p} dA_\beta(\textbf{z})dw\right)^\frac{1}{p}\\
	&=||Q||_{A^{p}_\alpha(\mathbb{D}^n,\textbf{z})\times H^p(\mathbb{T},w)}.
\end{align*}
Note that $Q$ is $m$-homogeneous, by Lemma \ref{Defant}, we complete the proof of the case $1<\alpha<\beta<\infty.$

For the case $1<\beta<\alpha<\infty,$ let
\[
	r_0=\sqrt{\frac{\beta p}{\alpha q}}<\sqrt{\frac{p}{q}}<1.
\] 
With similar arguments in the last case, we have
\begin{align*}
	||\textbf{T}_{\textbf{r}_0}Q||_{A^{q}_\alpha(\mathbb{D}^n,\textbf{z})\times H^{q}(\mathbb{T},w)}
	&=\left(\int_\mathbb{T}\int_{\mathbb{D}^n}|Q(\textbf{r}_0\textbf{z},r_0w)|^{q} dA_\beta(\textbf{z})dw\right)^\frac{1}{q}\\
	&\leqslant\left(\int_\mathbb{T}\left(\int_{\mathbb{D}^n}|Q(\textbf{z},r_0w)|^{p} dA_\beta(\textbf{z})\right)^{\frac{q}{p}}dw\right)^\frac{1}{q}\\
	&\leqslant\left(\int_{\mathbb{D}^n}\left(\int_\mathbb{T}|Q(\textbf{z},r_0w)|^{q} dw\right)^{\frac{p}{q}}dA_\beta(\textbf{z})\right)^\frac{1}{p}\\
	&\leqslant\left(\int_\mathbb{T}\int_{\mathbb{D}^n}|Q(\textbf{z},w)|^{p} dA_\beta(\textbf{z})dw\right)^\frac{1}{p}\\
	&=||Q||_{A^{p}_\alpha(\mathbb{D}^n,\textbf{z})\times H^p(\mathbb{T},w)}.
\end{align*}
This completes the proof of the case $1<\beta<\alpha<\infty$.

Now we show that the constant 
\[
	C(\alpha,\beta,p,q)=\sqrt{\frac{\alpha q}{\beta p}}
\]
is sharp. 
Let $n$ be a positive integer.
If one endows the polydisk $\mathbb{D}^n$ with Borel $\sigma$-field $\mathcal{B}(\mathbb{D}^n)$, then we get a probability space  $(\mathbb{D}^n, \mathcal{B}(\mathbb{D}^n), dA_\alpha(z_1)\dots dA_\alpha(z_n))$. Define 
$$z_i:\mathbb{D}^n\rightarrow \mathbb{C},\quad \textbf{z}\mapsto z_i, \quad 1\leqslant i\leqslant n,$$ 
then $\{z_i| 1\leqslant i\leqslant n\}$ are independent identically distributed random variables. Consider the $m$-homogeneous polynomial
\[
	p_{\alpha,n}(\textbf{z})=\sqrt{\frac{\alpha}{n}}\sum\limits_{i=1}^n z_i
\]
for given $n\in\mathbb{N}$ and $\alpha\geqslant 1.$  By the central limit theorem \cite{Legall}, the sequence of polynomials $p_{\alpha,n}$ conversges in distributions to the normal complex Gaussian random variable $G$ in distribution as $n$ approaches $\infty.$ Then for an fixed integer $m$
$$
	\lim\limits_{n\rightarrow\infty}\left\|p_{\alpha,n}^m\right\|_{A_\alpha^p(\mathbb{D}^n)}=\Gamma\left(\frac{mp}{2}+1\right).
$$
Note that $p_{1,n}=p_{\alpha,n}/\sqrt{\alpha}$.
Therefore,
$$\lim\limits_{n\rightarrow\infty}\left\|p_{1,n}^m\right\|_{A_\alpha^p(\mathbb{D}^n)}=\left(\frac{1}{\sqrt{\alpha}}\right)^m\Gamma\left(\frac{mp}{2}+1\right)^\frac{1}{p}.$$
Similarly, 
$$\lim\limits_{n\rightarrow\infty}\left\|p_{1,n}^m\right\|_{A_\beta^q(\mathbb{D}^n)}=\left(\frac{1}{\sqrt{\beta}}\right)^m\Gamma\left(\frac{mq}{2}+1\right)^\frac{1}{q}.$$
Hence,
\[
	\lim\limits_{n\rightarrow\infty}\frac{||p_{1,n}^m||_{A_\beta^q(\mathbb{D}^n)}}{||p_{1,n}^m||_{A_\alpha^p(\mathbb{D}^n)}}=\left(\sqrt{\frac{\alpha}{\beta}}\right)^m\frac{\Gamma\left(\frac{qm}{2}+1\right)^\frac{1}{q}}{\Gamma\left(\frac{pm}{2}+1\right)^\frac{1}{p}}.
\]
By the Stirling's formula
$$\sqrt{2\pi}x^{x+\frac{1}{2}}e^{-x}<\Gamma(x+1)<\sqrt{2\pi}x^{x+\frac{1}{2}}e^{-x+\frac{1}{12x}},$$ 
we have
\[
\frac{(2\pi)^{\frac{1}{2qm}}(\frac{qm}{2})^{\frac{1}{2}+\frac{1}{2qm}}e^{-\frac{1}{2}}}{(2\pi)^{\frac{1}{2pm}}(\frac{pm}{2})^{\frac{1}{2}+\frac{1}{2pm}}e^{-\frac{1}{2}+\frac{1}{6pm}}}
<\frac{\Gamma\left(\frac{qm}{2}+1\right)^\frac{1}{qm}}{\Gamma\left(\frac{pm}{2}+1\right)^\frac{1}{pm}}
<\frac{(2\pi)^{\frac{1}{2pm}}(\frac{pm}{2})^{\frac{1}{2}+\frac{1}{2pm}}e^{-\frac{1}{2}}}{(2\pi)^{\frac{1}{2qm}}(\frac{qm}{2})^{\frac{1}{2}+\frac{1}{2qm}}e^{-\frac{1}{2}+\frac{1}{6qm}}}
\]
Note that 
\[
\frac{(2\pi)^{\frac{1}{2qm}}(\frac{qm}{2})^{\frac{1}{2}+\frac{1}{2qm}}e^{-\frac{1}{2}}}{(2\pi)^{\frac{1}{2pm}}(\frac{pm}{2})^{\frac{1}{2}+\frac{1}{2pm}}e^{-\frac{1}{2}+\frac{1}{6pm}}}=e^\frac{1}{6pm}\frac{(qm\pi)^{\frac{1}{2qm}}}{(pm\pi)^{\frac{1}{2pm}}}\sqrt{\frac{q}{p}},
\]
and 
\[
\lim\limits_{m\rightarrow\infty}(qm)^\frac{1}{qm}=\lim\limits_{m\rightarrow\infty}e^{\frac{1}{qm}\ln qm}=1.
\]
It follows that 
\[
\lim\limits_{m\rightarrow\infty}\frac{\Gamma\left(\frac{qm}{2}+1\right)^\frac{1}{qm}}{\Gamma\left(\frac{pm}{2}+1\right)^\frac{1}{pm}}=\sqrt{\frac{q}{p}}.
\]
Consequently, 
\[
	\lim\limits_{n\rightarrow\infty}\frac{||p_{1,n}^m||_{A_\beta^q(\mathbb{D}^n)}}{||p_{1,n}^m||_{A_\alpha^p(\mathbb{D}^n)}}=\left(\sqrt{\frac{\alpha q}{\beta p}}\right)^m.
\]
Therefore, the constant 
\[
	C(\alpha,\beta,p,q)=\sqrt{\frac{\alpha q}{\beta p}}
\]
is sharp. 
Then the whole proof is complete now.

\end{document}